\documentclass{amsart}

\newtheorem{theorem}{Theorem}[section]
\newtheorem{corollary}[theorem]{Corollary}
\newtheorem{lemma}[theorem]{Lemma}

\theoremstyle{definition}
\newtheorem{definition}[theorem]{Definition}

\theoremstyle{remark}
\newtheorem{remark}[theorem]{Remark}

\numberwithin{equation}{section}

\newcommand{\Real}{\mathbb R}

\newcommand{\ip}[1]{\left\langle#1\right\rangle}
\newcommand{\set}[1]{\left\{#1\right\}}
\newcommand{\norm}[1]{\left\Vert#1\right\Vert}
\renewcommand{\vec}[1]{\boldsymbol{#1}}
\newcommand{\abs}[1]{\left\vert#1\right\vert}
\DeclareMathOperator{\esssup}{ess\,sup}

\begin{document}

\title[Gaussian estimates]{Gaussian estimates for fundamental solutions of
second order parabolic systems with time-independent coefficients}

\author{Seick Kim}
\address{Mathematics Department, University of Missouri, Columbia, Missouri 65211}
%\curraddr{Department of Mathematics and Statistics,
\email{seick@math.missouri.edu}

%    General info

%\date{January 1, 1994 and, in revised form, June 22, 1994.}

%\dedicatory{This paper is dedicated to our authors.}


\begin{abstract}
Auscher, McIntosh and Tchamitchian studied the heat kernels of second
order elliptic operators in divergence form with complex bounded measurable
coefficients on $\Real^n$.
In particular, in the case when $n=2$
they obtained Gaussian upper bound estimates
for the heat kernel without imposing further assumption on the coefficients.
We study the fundamental solutions of the systems of second order
parabolic equations in the divergence form with bounded, measurable,
time-independent coefficients, and
extend their results to the systems of parabolic equations.
\end{abstract}

\maketitle

%%%%%%%%%%%%%%%%%%%%%%%%%%%%%%%%%%%%%%%%%%%%%%%%%%%%%%%%%%%%%%%%%%%%%%%%%%%%%%%
\section{Introduction}
\label{sec:I}
%%%%%%%%%%%%%%%%%%%%%%%%%%%%%%%%%%%%%%%%%%%%%%%%%%%%%%%%%%%%%%%%%%%%%%%%%%%%%%%
In 1967, Aronson \cite{Aronson} proved Gaussian upper and lower bounds
for the fundamental solutions of parabolic equations in divergence form
with bounded measurable coefficients.
To establish the Gaussian lower bound Aronson made use of the Harnack
inequality for nonnegative solutions which was proved by Moser in 1964
(see \cite{Moser}).
Related to Moser's parabolic Harnack inequality, we should mention
Nash's earlier paper \cite{Nash} where the H\"{o}lder
continuity of weak solutions to parabolic equations in divergence form was
established.
In 1985, Fabes and Stroock \cite{FS} showed that the idea of Nash
could be used to establish a Gaussian upper and lower bound on the fundamental
solution.
They showed that actually
such Gaussian estimates could be used to prove Moser's Harnack inequality.
We note that Aronson also obtained Gaussian upper bound estimates of 
the fundamental solution without using Moser's Harnack inequality.

In \cite{Auscher}, Auscher proposed a new proof of Aronson's Gaussian
upper bound estimates for the fundamental solution of second order
parabolic equations with time-independent coefficients.
His method relies crucially on the assumption that the coefficients are
time-independent and thus it does not exactly reproduce Aronson's result,
which is valid even for the time-dependent coefficients case.
However, his method is interesting in the sense that it
carries over to equations with complex coefficients provided that the
complex coefficients are a small perturbation of real coefficients.
Along with this direction,
Auscher, McIntosh and Tchamitchian also showed that
the heat kernel of second order elliptic operators in divergence form
with complex bounded measurable coefficients in the two dimensional space
has a Gaussian upper bound
(see \cite{AMT} and also \cite{AT}).

We would like to point out that
a parabolic equation with complex coefficients is, in fact,
a special case of a system of parabolic equations.
From this point of view,
Hofmann and the author showed that the fundamental solution
of a parabolic system
has an upper Gaussian bound if the system is
a small perturbation of a diagonal system, which, in particular,
generalized the result of Auscher mentioned above to
the time-dependent coefficients case (see \cite{HK}).
However, the above mentioned result of Auscher, McIntosh and Tchamitchian
regarding the heat kernel of two dimensional elliptic
operators with complex coefficients
does not follow directly from our result.

One of the main goals of this article is to provide a proof that
weak solutions of the parabolic system of divergence type
with time-independent coefficients associated to an
elliptic system in two dimensions enjoy the parabolic local boundedness
property and to show that its fundamental solution has a
Gaussian upper bound. More generally, we show that if weak solutions
of an elliptic system satisfy H\"{o}lder estimates at every scale,
then weak solutions of the corresponding parabolic system with time-independent
coefficients also satisfies similar parabolic H\"{o}lder estimates
from which, in particular,
the parabolic local boundedness property follows easily.
Also, such an argument allows one to derive 
H\"{o}lder continuity estimates for weak solutions of parabolic equations
with time-independent coefficients directly from
De Giorgi's theorem \cite{DG57}
on elliptic equations, bypassing Moser's parabolic Harnack inequality.
In fact, this is what Auscher really proved in the
setting of complex coefficients equations by using a functional calculus method
(see \cite{Auscher} and also \cite{AQ}, \cite{AT}).
Even in those complex coefficients settings,
we believe that our approach is much more straightforward and thus
appeals to wider readership.

Finally, we would like to point out that
in this article,
we are mainly interested in global estimates and that
we do not attempt to treat, for example,
the systems with lower order terms, etc.
However, let us also mention that, with some extra technical details,
our methods carry over to those cases as well as to the systems of higher order;
see e.g. \cite{AQ}, \cite{AT} for the details, and also Remark~\ref{rmk:local}.

The remaining sections are organized in the following way.
In Section~\ref{sec:N} we give notations, definitions, and some known facts.
We state the main results in Section~\ref{sec:M}
and give the proofs in Section~\ref{sec:P}.

%%%%%%%%%%%%%%%%%%%%%%%%%%%%%%%%%%%%%%%%%%%%%%%%%%%%%%%%%%%%%%%%%%%%%%%%%%%%%%%
\section{Notation and definitions}
\label{sec:N}
%%%%%%%%%%%%%%%%%%%%%%%%%%%%%%%%%%%%%%%%%%%%%%%%%%%%%%%%%%%%%%%%%%%%%%%%%%%%%%%
\subsection{Geometric notation}
%%%%%%%%%%%%%%%%%%%%%%%%%%%%%%%%%%%%%%%%%%%%%%%%%%%%%%%%%%%%%%%%%%%%%%%%%%%%%%%
\begin{enumerate}
\item
$\Real^n=\text{$n$-dimensional real Euclidean space.}$
\item
$x=(x_1,\cdots,x_n)$ is an arbitrary point of $\Real^{n}$.
\item
$X=(x,t)$ denotes an arbitrary point in $\Real^{n+1}$,
where $x\in\Real^n$ and $t\in\Real$.
\item
$B_r(x)=\set{y\in\Real^n:\abs{y-x}<r}$ is an open ball in $\Real^n$ with center
$x$ and radius $r>0$.
We sometimes drop the reference point $x$ and 
write $B_r$ for $B_r(x)$ if there is no danger of confusion.
\item
$Q_r(X)=\set{(y,s)\in\Real^{n+1}: \abs{y-x}<r\text{ and } t-r^2<s<t}$.
We sometimes drop the reference point $X$ and write $Q_r$ for $Q_r(X)$.
\item
$Q^{*}_r(X)=\set{(y,s)\in\Real^{n+1}: \abs{y-x}<r\text{ and } t<s<t+r^2}$.
\item
$Q_{r,s}(X)=\set{(y,s)\in Q_r(X)}$; i.e.,
$Q_{r,s}(X)=B_r(x)\times\set{s}$ if $s\in(t-r^2,t)$ and
$Q_{r,s}(X)=\emptyset$ otherwise.
We sometimes drop the reference point $X$ and write $Q_{r,s}$ for $Q_{r,s}(X)$.
\item
For a cylinder $Q=\Omega\times (a,b)\subset \Real^{n+1}$,
$\partial_P Q$ denotes its parabolic boundary, namely,
$\partial_P Q=\partial\Omega\times (a,b)\cup \overline{\Omega}\times\set{a}$,
where $\partial\Omega$ is the usual topological boundary of
$\Omega\subset\Real^n$ and $\overline\Omega$ is its closure.
\end{enumerate}

%%%%%%%%%%%%%%%%%%%%%%%%%%%%%%%%%%%%%%%%%%%%%%%%%%%%%%%%%%%%%%%%%%%%%%%%%%%%%%%
\subsection{Notation for functions and their derivatives}
%%%%%%%%%%%%%%%%%%%%%%%%%%%%%%%%%%%%%%%%%%%%%%%%%%%%%%%%%%%%%%%%%%%%%%%%%%%%%%%
\begin{enumerate}
\item For a mapping from  $\Omega\subset\Real^n$ to $\Real^N$,
we write $\vec{f}(x)=(f^1(x),\ldots,f^N(x))^T$ as a column vector.
\item
$\overline{f}_{Q}=\frac{1}{\abs{Q}}\int_{Q}f$,
where $\abs{Q}$ denotes the volume of $Q$.
\item
$u_t=\partial u/\partial t$.
\item
$D_{x_i} u= D_i u= u_{x_i}=\partial u/\partial x_i$.
\item
$D u=(u_{x_1},\ldots,u_{x_n})^T$ is the spatial gradient of $u=u(x,t)$.
\item
For $\vec{f}=(f^1,\ldots,f^N)^T$,
$D\vec{f}=(Df^1,\ldots,Df^N)$; that is $D\vec{f}$ is the $n\times N$ matrix
whose $i$-th column is $Df^i$.
\end{enumerate}

%%%%%%%%%%%%%%%%%%%%%%%%%%%%%%%%%%%%%%%%%%%%%%%%%%%%%%%%%%%%%%%%%%%%%%%%%%%%%%%
\subsection{Function spaces}
%%%%%%%%%%%%%%%%%%%%%%%%%%%%%%%%%%%%%%%%%%%%%%%%%%%%%%%%%%%%%%%%%%%%%%%%%%%%%%%
\begin{enumerate}
\item
For $\Omega\subset\Real^n$ and $p\ge 1$, $L^p(\Omega)$ denotes the space
of functions with the following norms:
\begin{equation*}
\norm{u}_{L^p(\Omega)}=\left(\int_\Omega\abs{u(x)}^p\,dx\right)^{1/p}\quad\text{and}\quad
\norm{u}_{L^\infty(\Omega)}=\esssup_\Omega\abs{u}.
\end{equation*}
\item
$C^{\mu}(\Omega)$ denotes the space of
functions that are H\"{o}lder continuous with the exponent $\mu\in (0,1]$,
and
\begin{equation*}
[u]_{C^{\mu}(\Omega)}
=\sup_{x\neq x'\in \Omega}\frac{\abs{u(x)-u(x')}}{\abs{x-x'}^\mu}<\infty.
\end{equation*}
\item
The Morrey space $M^{2,\mu}(\Omega)$ is the set of all functions
$u\in L^2(\Omega)$ such that
\begin{equation*}
\norm{u}_{M^{2,\mu}(\Omega)}=\sup_{B_\rho(x)\subset \Omega}
\left(\rho^{-\mu}\int_{B_\rho(x)}\abs{u}^2\right)^{1/2}<\infty.
\end{equation*}
\item
$C^{\mu}_{P}(Q)$ denotes the space of
functions defined on $Q\subset\Real^{n+1}$ such that
\begin{equation*}
[u]_{C^{\mu}_{P}(Q)}
=\sup_{X\neq X'\in Q}\frac{\abs{u(X)-u(X')}}{d_P(X,X')^\mu}<\infty,
\end{equation*}
where $d_P(X,X')=\max\left(\abs{x-x'},\sqrt{\abs{t-t'}}\right)$.
%where $d_P(X,X')=\max\left(\abs{x-x'},\abs{t-t'}^{1/2}\right)$.
\end{enumerate}

%%%%%%%%%%%%%%%%%%%%%%%%%%%%%%%%%%%%%%%%%%%%%%%%%%%%%%%%%%%%%%%%%%%%%%%%%%%%%%%
\subsection{Elliptic and parabolic systems and their adjoints}
%%%%%%%%%%%%%%%%%%%%%%%%%%%%%%%%%%%%%%%%%%%%%%%%%%%%%%%%%%%%%%%%%%%%%%%%%%%%%%%
\begin{definition}
We say that the coefficients $A^{\alpha\beta}_{ij}(x)$
satisfy the uniform ellipticity condition
if there exist numbers $\nu_0, M_0>0$ such that for all $x\in\Real^n$ we have
\begin{equation}
\label{eqn:para}
\ip{\vec{A}^{\alpha\beta}(x)\vec{\xi}_\beta,\vec{\xi}_\alpha}\ge
\nu_0\abs{\vec{\xi}}^2\quad\text{and }
\abs{\ip{\vec{A}^{\alpha\beta}(x)\vec{\xi}_\beta,\vec{\eta}_\alpha}}\le
M_0\abs{\vec{\xi}}\abs{\vec{\eta}},
\end{equation}
where we used the following notation.
\begin{enumerate}
\item
For $\alpha,\beta=1,\ldots,n$,
$\vec{A}^{\alpha\beta}(x)$ are
$N\times N$ matrices with $(i,j)$-entries $A^{\alpha\beta}_{ij}(x)$.
\item
$\vec{\xi}_\alpha=(\xi_\alpha^1,\cdots,\xi_\alpha^N)^T$
and $\abs{\vec{\xi}}^2=\sum\limits_{\alpha=1}^n\sum\limits_{i=1}^N\abs{\xi_\alpha^i}^2$.
\item
$\ip{\vec{A}^{\alpha\beta}(x)\vec{\xi}_\beta,\vec{\eta}_\alpha}
=\sum\limits_{\alpha,\beta=1}^n \sum\limits_{i,j=1}^N
A_{ij}^{\alpha\beta}(x)\xi_\beta^j \eta_\alpha^i$.
\end{enumerate}
We emphasize that we do not assume that the coefficients are symmetric.
\end{definition}

\begin{definition}
We say that a system of $N$ equations on $\Real^n$
\begin{equation*}
\sum_{j=1}^N\sum_{\alpha,\beta=1}^n
D_{x_\alpha}(A^{\alpha\beta}_{ij}(x) D_{x_\beta}u^j)=0\qquad
(i=1,\ldots,N)
\end{equation*}
is elliptic if the coefficients satisfy
the uniform ellipticity condition.
We often write the above system in a vector form
\begin{equation}
\label{eqn:E-01}
L\vec{u}:=\sum_{\alpha,\beta=1}^n D_\alpha(\vec{A}^{\alpha\beta}(x)
D_\beta\vec{u})=0,
\quad\vec{u}=(u^1\ldots,u^N)^T.
\end{equation}
The adjoint system of \eqref{eqn:E-01} is given by
\begin{equation}
\label{eqn:E-02}
L^{*}\vec{u}:=\sum_{\alpha,\beta=1}^n
D_\alpha\left((\vec{A}^{\alpha\beta}){}^{*}(x) D_\beta\vec{u}\right)=0,
\end{equation}
where $(\vec{A}^{\alpha\beta}){}^{*}=(\vec{A}^{\beta\alpha})^T$,
the transpose of $\vec{A}^{\beta\alpha}$.
\end{definition}

\begin{definition}
We say that a system of $N$ equations on $\Real^{n+1}$
\begin{equation*}
u^i_t-\sum_{j=1}^N\sum_{\alpha,\beta=1}^n
D_{x_\alpha}(A^{\alpha\beta}_{ij}(x) D_{x_\beta}u^j)=0\qquad
(i=1,\ldots,N)
\end{equation*}
is parabolic if the (time-independent) coefficients satisfy
the uniform ellipticity condition.
We often write the above system in a vector form
\begin{equation}
\label{eqn:P-01}
\vec{u}_t-L\vec{u}
:=\vec{u}_t-\sum_{\alpha,\beta=1}^n D_\alpha(\vec{A}^{\alpha\beta}(x)
D_\beta\vec{u})=0.
\end{equation}
The adjoint system of \eqref{eqn:P-01} is given by
\begin{equation}
\label{eqn:P-02}
\vec{u}_t+L^{*}\vec{u}
:=\vec{u}_t+\sum_{\alpha,\beta=1}^n
D_\alpha\left((\vec{A}^{\alpha\beta}){}^{*}(x) D_\beta\vec{u}\right)=0,
\end{equation}
where $(\vec{A}^{\alpha\beta}){}^{*}=(\vec{A}^{\beta\alpha})^T$,
the transpose of $\vec{A}^{\beta\alpha}$.
\end{definition}

%%%%%%%%%%%%%%%%%%%%%%%%%%%%%%%%%%%%%%%%%%%%%%%%%%%%%%%%%%%%%%%%%%%%%%%%%%%%%%%
\subsection{Weak solutions}
%%%%%%%%%%%%%%%%%%%%%%%%%%%%%%%%%%%%%%%%%%%%%%%%%%%%%%%%%%%%%%%%%%%%%%%%%%%%%%%
In this article, the term ``weak solution'' is used in a rather abusive way.
To avoid unnecessary technicalities,
we may assume that all the coefficients involved are smooth so that
all weak solutions are indeed classical solutions.
However, this extra smoothness assumption will not
be used quantitatively in our estimates.
This is why we shall make clear the dependence of constants.
%--HERE
\begin{enumerate}
\item
We say that $\vec{u}$ is a weak solution of \eqref{eqn:E-01} in $\Omega\subset\Real^n$ if $\vec{u}$ is a (classical) solution of \eqref{eqn:E-01} in $\Omega$ and $\vec{u}, D\vec{u}\in L^2(\Omega)$.
\item
We say that $\vec{u}$ is a weak solution of \eqref{eqn:P-01} in a cylinder
$Q=\Omega\times (a,b)\subset\Real^{n+1}$ if $\vec{u}$
is a (classical) solution of \eqref{eqn:E-01} in $Q$ and
$\vec{u}, D\vec{u}\in L^2(Q)$, $\vec{u}(\cdot,t)\in L^2(\Omega)$ for all
$a\le t\le b$, and $\sup_{a\le t\le b}
\norm{\vec{u}(\cdot,t)}_{L^2(\Omega)}<\infty$.
\end{enumerate}
%%%%%%%%%%%%%%%%%%%%%%%%%%%%%%%%%%%%%%%%%%%%%%%%%%%%%%%%%%%%%%%%%%%%%%%%%%%%%%%
\subsection{Fundamental solution}
%%%%%%%%%%%%%%%%%%%%%%%%%%%%%%%%%%%%%%%%%%%%%%%%%%%%%%%%%%%%%%%%%%%%%%%%%%%%%%%
By a fundamental solution (or fundamental matrix) $\vec{\Gamma}(x,t;y)$
of the parabolic system \eqref{eqn:P-01}
we mean an $N\times N$ matrix of functions defined for $t>0$ which,
as a function of $(x,t)$, is a solution of \eqref{eqn:P-01}
(i.e., each column is a solution of \eqref{eqn:P-01}),
and is such that
\begin{eqnarray}
\lim_{t\downarrow 0}\int_{\Real^n}\vec{\Gamma}(x,t;y)\vec{f}(y)\,dy
=\vec{f}(x)
\end{eqnarray}
for any bounded continuous function $\vec{f}=(f^1,\ldots,f^N)^T$,
where $\vec{\Gamma}(x,t;y)\vec{f}(y)$ denotes the usual matrix multiplication.

%%%%%%%%%%%%%%%%%%%%%%%%%%%%%%%%%%%%%%%%%%%%%%%%%%%%%%%%%%%%%%%%%%%%%%%%%%%%%%%
\subsection{Notation for estimates}
%%%%%%%%%%%%%%%%%%%%%%%%%%%%%%%%%%%%%%%%%%%%%%%%%%%%%%%%%%%%%%%%%%%%%%%%%%%%%%%
We employ the letter $C$ to denote a universal constant usually depending
on the dimension and ellipticity constants.
It should be understood that $C$ may vary from line to line.
We sometimes write $C=C(\alpha,\beta,\ldots)$ to emphasize the 
dependence on the prescribed quantities $\alpha,\beta,\ldots$.

%%%%%%%%%%%%%%%%%%%%%%%%%%%%%%%%%%%%%%%%%%%%%%%%%%%%%%%%%%%%%%%%%%%%%%%%%%%%%%%
\subsection{Some preliminary results and known facts}
%%%%%%%%%%%%%%%%%%%%%%%%%%%%%%%%%%%%%%%%%%%%%%%%%%%%%%%%%%%%%%%%%%%%%%%%%%%%%%%
\begin{lemma}[Energy estimates]
\label{lem:P-03}
Let $\vec{u}$ be a weak solution of \eqref{eqn:P-01} in $Q_R=Q_R(X)$.
Then for $0<r<R$, we have
\begin{equation*}
\sup_{t-r^2\le s\le t}
\int_{Q_{r,s}}\abs{\vec{u}(\cdot,s)}^2+
\int_{Q_r}\abs{D\vec{u}}^2\le \frac{C}{(R-r)^2} \int_{Q_R}\abs{\vec{u}}^2.
\end{equation*}
\end{lemma}
\begin{proof}
See e.g., \cite[Lemma 2.1, p. 139]{LSU}.
\end{proof}

\begin{lemma}[Parabolic Poincar\'{e} inequality]
\label{lem:P-02}
Let $\vec{u}$ be a weak solution of \eqref{eqn:P-01} in $Q_R=Q_R(X)$.
Then there is some constant $C=C(n,M_0)$ such that
\begin{equation*}
\int_{Q_R}\abs{\vec{u}-\overline{\vec{u}}_{Q_R}}^2\le C R^2
\int_{Q_R}\abs{D\vec{u}}^2.
\end{equation*}
\end{lemma}
\begin{proof}
See e.g., \cite[Lemma 3]{Struwe}.
\end{proof}

\begin{lemma}
\label{lem:P-04}
Let $Q_{2R}=Q_{2R}(X_0)$ be a cylinder in $\Real^{n+1}$.
Suppose $\vec{u}\in L^2(Q_{2R})$ and
there are positive constants $\mu\le 1$ and $M$ such that
for any $X\in Q_R$ and any $r\in (0,R)$ we have
\begin{equation*}
\int_{Q_r(X)}\abs{\vec{u}-\overline{\vec{u}}_{Q_{r}(X)}}^2\le M^2r^{n+2+2\mu}.
\end{equation*}
Then $\vec{u}$ is H\"{o}lder continuous in $Q_R$ with the exponent $\mu$
and $[\vec{u}]_{C^{\mu}_P(Q_R)}\le C(n,\mu)M$.
\end{lemma}
\begin{proof}
See e.g., \cite[Lemma 4.3, p. 50]{L}.
\end{proof}

\begin{definition}[Local boundedness property]
We say that the system \eqref{eqn:P-01}
satisfies the local boundedness property for weak solutions
if there is a constant
$M$ such that all weak solutions $\vec{u}$ of \eqref{eqn:P-01}
in $Q_{2r}(X)$ satisfy the estimates
\begin{equation*}
\sup_{Q_r(X)}\abs{\vec{u}}\le M\left(\frac{1}{\abs{Q_{2r}}}
\int_{Q_{2r}(X)}\abs{\vec{u}}^2 \right)^{1/2}.
\end{equation*}
Similarly, we say that the adjoint system \eqref{eqn:P-02}
satisfies the local boundedness property if the corresponding estimates hold
for weak solutions $\vec{u}$ of \eqref{eqn:P-02} in $Q_{2r}^{*}(X)$.
\end{definition}

\begin{theorem}[Theorem~1.1, \cite{HK}]
\label{thm:P-01}
Assume that the system \eqref{eqn:P-01} and its adjoint system
\eqref{eqn:P-02}
satisfy the local boundedness property for weak solutions.
Then the fundamental solution of the system \eqref{eqn:P-01} has an
upper bound
\begin{equation}
\abs{\vec{\Gamma}(x,t;y)}_{op}\le C_0
t^{-{n/2}}\exp\left(-\frac{k_0\abs{x-y}^2}{t}\right),
\end{equation}
where $\abs{\vec{\Gamma}(x,t;y)}_{op}$ denotes the operator norm
of the fundamental matrix $\vec{\Gamma}(x,t;y)$.
Here, $C_0=C_0(n,\nu_0,M_0,M)$ and $k_0=k_0(\nu_0,M_0)$.
\end{theorem}

%%%%%%%%%%%%%%%%%%%%%%%%%%%%%%%%%%%%%%%%%%%%%%%%%%%%%%%%%%%%%%%%%%%%%%%%%%%%%%%
\section{Main results}
\label{sec:M}
%%%%%%%%%%%%%%%%%%%%%%%%%%%%%%%%%%%%%%%%%%%%%%%%%%%%%%%%%%%%%%%%%%%%%%%%%%%%%%%
\begin{definition}
We say that an elliptic system \eqref{eqn:E-01}
satisfies the H\"{o}lder estimates for weak solutions at every scale
if there exist constants $\mu_0>0$ and $H_0$ such that
all weak solutions $\vec{u}$ of the system
in $B_{2r}=B_{2r}(x_0)$ satisfy the following estimates
\begin{equation}
\label{eqn:M-04}
[\vec{u}]_{C^{\mu_0}(B_r)}\le H_0 r^{-(n/2+\mu_0)}
\norm{\vec{u}}_{L^2(B_{2r})}.
\end{equation}
Similarly, we say that a parabolic system \eqref{eqn:P-01}
satisfies H\"{o}lder estimates for weak solutions at every scale
if there exist constants $\mu_1>0$ and $H_1$ such that
all weak solutions $\vec{u}$ of the system
in $Q_{2r}=Q_{2r}(X_0)$ satisfy the following estimates
\begin{equation}
\label{eqn:M-05}
[\vec{u}]_{C^{\mu_1}_P(Q_r)}\le H_1 r^{-(n/2+1+\mu_1)}
\norm{\vec{u}}_{L^2(Q_{2r})}.
\end{equation}
\end{definition}

\begin{remark}
Elliptic systems with constant coefficients satisfy the above property,
and in that case, the ellipticity condition \eqref{eqn:para}
can be weakened and replaced by the Legendre-Hadamard condition.
De Giorgi's theorem \cite{DG57} states that the property is satisfied if $N=1$.
The property is also satisfied if $n=2$ and
it is due to Morrey (see Corollary~\ref{thm:M-04}).
Some other examples include, for instance,
a certain three dimensional elliptic system which was studied
by Kang and the author in \cite{KK}.
\end{remark}

We shall prove the following main results in this paper:
\begin{theorem}
\label{thm:M-02}
If an elliptic system \eqref{eqn:E-01} satisfies
the H\"{o}lder estimates for weak solutions at every scale, then
the corresponding parabolic system \eqref{eqn:P-01} with time-independent
coefficients also satisfies the H\"{o}lder estimates for weak solutions at
every scale.
\end{theorem}

\begin{theorem}
\label{thm:M-03}
Suppose that the elliptic system \eqref{eqn:E-01}
and its adjoint system \eqref{eqn:E-02} defined on $\Real^n$
both satisfy the H\"{o}lder estimates for weak solutions at every scale with
constants  $\mu_0, H_0$.
Let $\vec{\Gamma}(x,t;y)$ be the fundamental solution of
the parabolic system
\eqref{eqn:P-01}
with the time-independent coefficients
associated to the elliptic system \eqref{eqn:E-01}.
Then $\vec{\Gamma}(x,t;y)$ has an upper bound
\begin{equation}
\label{bound}
\abs{\vec{\Gamma}(x,t;y)}_{op}\le C_0
t^{-n/2}\exp\left(-\frac{k_0\abs{x-y}^2}{t}\right),
\end{equation}
where $C_0=C_0(n,\nu_0,M_0,\mu_0,H_0)$ and $k_0=k_0(\nu_0,M_0)$.
Here, $\abs{\vec{\Gamma}(x,t;y)}_{op}$ denotes the operator norm of fundamental
matrix $\vec{\Gamma}(x,t;y)$.
\end{theorem}

\begin{remark}
\label{rmk:local}
We would like to point out that \eqref{bound} is a global estimate.
Especially, the bound \eqref{bound} holds for all time $t>0$.
Suppose that the elliptic system \eqref{eqn:E-01} and its adjoint system
\eqref{eqn:E-02} enjoy the H\"{o}lder estimates
for weak solutions up to a fixed scale $R_0$; that is,
there is a number $R_0>0$ such that if $\vec{u}$ is a weak solution of
either \eqref{eqn:E-01} or \eqref{eqn:E-02} in $B_r=B_r(x)$
with $0<r\le R_0$, then
$\vec{u}$ is H\"{o}lder continuous and satisfies
\begin{equation*}
[\vec{u}]_{C^{\mu_0}(B_r)}\le H_0 r^{-(n/2+\mu_0)}
\norm{\vec{u}}_{L^2(B_{2r})}.
\end{equation*}
Then, the statement regarding the bound \eqref{bound}
for the fundamental solution should be localized as follows:
For any given $T>0$, there are constants
$k_0=k_0(\nu_0,M_0)$
and $C_0=C_0(n,\nu_0,M_0,\mu_0,H_0,R_0,T)$
such that \eqref{bound} holds for $0<t\le T$.
\end{remark}

\begin{corollary}
\label{thm:M-04}
Let $\vec{\Gamma}(x,t;y)$ be the fundamental solution
of the parabolic system
\eqref{eqn:P-01}
with time-independent coefficients
associated to an elliptic
system \eqref{eqn:E-01} defined on $\Real^2$.
Then $\vec{\Gamma}(x,t;y)$ has an upper bound
\eqref{bound} with the constants $C_0, k_0$
depending only on the ellipticity constants
$\nu_0, M$.
\end{corollary}
\begin{proof}
First, let us recall the well known theorem of Morrey
which states that any two dimensional elliptic system
\eqref{eqn:E-01} with bounded measurable coefficients
satisfies the H\"{o}lder estimates
for weak solutions at every scale, with the constants $\mu_0, H_0$
depending only on the ellipticity constants
(see, \cite[pp. 143--148]{Morrey}).
Next, note that the ellipticity constants $\nu_0, M_0$ in
\eqref{eqn:para} remain unchanged for
$\tilde{A}{}^{\alpha\beta}_{ij}(x)=A^{\beta\alpha}_{ji}(x)$. 
Therefore, the corollary is an immediate consequence of Theorem~\ref{thm:M-03}.
\end{proof}

\begin{remark}
In fact, the converse of Theorem~\ref{thm:P-01} is also true
(see \cite[Theorem~1.2]{HK}).
Therefore, in order to extend the above corollary to
the parabolic system with time-dependent coefficients, one needs to show
that the system satisfies the local boundedness property for weak solutions.
Unfortunately, we do not know whether it is true or not if the coefficients
are allowed to depend on the time variable.
If $n\ge 3$, it is not true in general,
even for the time-independent coefficients case
since there is a famous counter-example due to De Giorgi (see \cite{DG68}).
\end{remark}
%%%%%%%%%%%%%%%%%%%%%%%%%%%%%%%%%%%%%%%%%%%%%%%%%%%%%%%%%%%%%%%%%%%%%%%%%%%%%%%
\section{Proof of Main Results}
\label{sec:P}
%%%%%%%%%%%%%%%%%%%%%%%%%%%%%%%%%%%%%%%%%%%%%%%%%%%%%%%%%%%%%%%%%%%%%%%%%%%%%%%
\subsection{Some technical lemmas and proofs}
\begin{lemma}
\label{lem:M-01}
If $\vec{u}$ is a weak solution of
the parabolic system with time-independent coefficients \eqref{eqn:P-01}
in $Q_{R}=Q_{R}(X_0)$, then
$\vec{u}_t\in L^2(Q_{r})$ for $r<R$ and 
satisfies the estimates
\begin{equation}
\norm{\vec{u}_t}_{L^2(Q_{r})} \le C(R-r)^{-1} \norm{D\vec{u}}_{L^2(Q_{R})}.
\end{equation}
In particular,
if $\vec{u}$ is a weak solution of \eqref{eqn:P-01}
in $Q_{2r}$, then the above estimates together with the energy estimates
yield
\begin{equation}
\label{eqn:M-00}
\norm{\vec{u}_t}_{L^2(Q_{r})} \le Cr^{-2}
\norm{\vec{u}}_{L^2(Q_{2r})}.
\end{equation}
\end{lemma}

\begin{proof}
We first note that if the coefficients are symmetric,
(i.e., $A_{ij}^{\alpha\beta}=A_{ji}^{\beta\alpha}$)
this is a well known result; a proof for such a case
is found, for example,
in \cite[pp. 172--181]{LSU} or in \cite[pp. 360--364]{Evans}.
However, the standard proof does not carry over to the non-symmetric
coefficients case and for that reason, we provide a self-contained proof here.

Fix positive numbers $\sigma,\tau$ such that $\sigma<\tau\le R$.
Let $\zeta$ be a smooth cut-off function such that $\zeta\equiv 1$
in $Q_\sigma$, vanishes near $\partial_P Q_{\tau}$, and
satisfies
%(note that $\tau^2-\sigma^2>(\tau-\sigma)^2$)
\begin{equation*}
\label{eqn:K-01}
0\le \zeta \le 1\quad\text{and}\quad
\abs{\zeta_t}+\abs{D\zeta}^2 \le C(\tau-\sigma)^{-2}.
\end{equation*}
Note that on each slice $Q_{\tau,s}$, we have
\begin{equation*}
\begin{split}
0&=\int_{Q_{\tau,s}}\left(\vec{u}_t-D_{\alpha}(\vec{A}^{\alpha\beta}D_\beta
\vec{u})\right)\cdot\zeta^2\vec{u}_t\\
&=\int_{Q_{\tau,s}}\zeta^2\abs{\vec{u}_t}^2+\int_{Q_{\tau,s}}\zeta^2
\ip{\vec{A}^{\alpha\beta}D_\beta\vec{u},D_\alpha\vec{u}_t}+
\int_{Q_{\tau,s}}
2\zeta\ip{\vec{A}^{\alpha\beta}D_\beta\vec{u},D_\alpha\zeta\vec{u}_t}.
\end{split}
\end{equation*}
Therefore, we find by using the Cauchy-Schwarz inequality that
\begin{equation*}
\begin{split}
\int_{Q_{\tau,s}}\zeta^2\abs{\vec{u}_t}^2
&\le C\int_{Q_{\tau,s}}\zeta^2\abs{D\vec{u}}\abs{D\vec{u}_t}
+C\int_{Q_{\tau,s}}\zeta\abs{D\vec{u}}\abs{D\zeta}\abs{\vec{u}_t}\\
&\le \frac{\epsilon}{2}\int_{Q_{\tau,s}}\zeta^2\abs{D\vec{u}_t}^2+
\frac{C}{\epsilon}\int_{Q_{\tau,s}}\zeta^2\abs{D\vec{u}}^2+
C\int_{Q_{\tau,s}}\abs{D\zeta}^2\abs{D\vec{u}}^2 \\
&+\frac{1}{2}\int_{Q_{\tau,s}}\zeta^2\abs{\vec{u}_t}^2.
\end{split}
\end{equation*}
Thus we have
\begin{equation}
\label{eqn:K-02}
\qquad
\int_{Q_\tau}\zeta^2\abs{\vec{u}_t}^2
\le \epsilon\int_{Q_\tau}\zeta^2\abs{D\vec{u}_t}^2+
\frac{C}{\epsilon}\int_{Q_\tau}\zeta^2\abs{D\vec{u}}^2+
C\int_{Q_\tau}\abs{D\zeta}^2\abs{D\vec{u}}^2.
\end{equation}
Since $\vec{u}_t$
also satisfies \eqref{eqn:P-01}, the energy estimates
%together with \eqref{eqn:K-01}
yield
\begin{equation}
\label{eqn:K-03}
\int_{Q_{\tau}}\zeta^2\abs{D\vec{u}_t}^2
\le \frac{C_0}{(\tau-\sigma)^2}\int_{Q_{\tau}}\abs{\vec{u}_t}^2.
\end{equation}
This is the part where we exploit the assumption
that the coefficients are time-independent.
Combining  \eqref{eqn:K-02} and \eqref{eqn:K-03}, we have
\begin{equation*}
\int_{Q_{\sigma}}\abs{\vec{u}_t}^2
\le \frac{C_0\epsilon}{(\tau-\sigma)^2}\int_{Q_{\tau}}\abs{\vec{u}_t}^2+
\frac{C}{\epsilon}\int_{Q_\tau}\abs{D\vec{u}}^2+
\frac{C}{(\tau-\sigma)^2}\int_{Q_{\tau}}\abs{D\vec{u}}^2.
\end{equation*}
If we set $\epsilon=(\tau-\sigma)^2/2C_0$, we finally obtain
\begin{equation*}
\int_{Q_{\sigma}}\abs{\vec{u}_t}^2
\le \frac{1}{2}\int_{Q_{\tau}}\abs{\vec{u}_t}^2+
\frac{C}{(\tau-\sigma)^2}\int_{Q_{\tau}}\abs{D\vec{u}}^2.
\end{equation*}
Here, we emphasize that
$C$ is a constant independent of $\sigma,\tau$.
Then by a standard iteration argument
(see e.g. \cite[Lemma~{3.1}, pp. 161]{Giaq83}),
we have
\begin{equation}
\label{eqn:K-05}
\int_{Q_{r}}\abs{\vec{u}_t}^2
\le \frac{C}{(R-r)^2}\int_{Q_{R}}\abs{D\vec{u}}^2
\quad\text{for } 0<r<R.
\end{equation}
The proof is complete.
\end{proof}

\begin{lemma}
\label{lem:M-02}
If $\vec{u}$ is a weak solution of
the parabolic system with time-independent coefficients \eqref{eqn:P-01}
in $Q_{2r}=Q_{2r}(X_0)$, then
$D\vec{u}(\cdot,s), \vec{u}_t(\cdot,s)\in L^2(Q_{r,s})$
for all $s\in[t_0-r^2,t_0]$,
and satisfy the following estimates uniformly in $s\in[t_0-r^2,t_0]$.
\begin{eqnarray}
\label{eqn:M-11}
\norm{D\vec{u}(\cdot,s)}_{L^2(Q_{r,s})}
\le C r^{-2} \norm{\vec{u}}_{L^2(Q_{2r})}, \\
\label{eqn:M-12}
\norm{\vec{u}_t(\cdot,s)}_{L^2(Q_{r,s})}
\le C r^{-3} \norm{\vec{u}}_{L^2(Q_{2r})}.
\end{eqnarray}
\end{lemma}

\begin{proof}
By the energy estimates applied to $\vec{u}_t$ we obtain
\begin{equation}
\label{eqn:M-13}
\sup_{t_0-r^2\le s\le t_0}\int_{Q_{r,s}}\abs{\vec{u}_t(\cdot,s)}^2
\le \frac{C}{r^2}\int_{Q_{3r/2}}
\abs{\vec{u}_t}^2.
\end{equation}
On the other hand, the estimates \eqref{eqn:K-05} and 
the energy estimates (this time, applied to $\vec{u}$ itself) yield
\begin{eqnarray}
\label{eqn:K-12}
\int_{Q_{3r/2}} \abs{\vec{u}_t}^2 \le
\frac{C}{r^2} \int_{Q_{7r/4}} \abs{D\vec{u}}^2\le
\frac{C}{r^4} \int_{Q_{2r}}\abs{\vec{u}}^2.
\end{eqnarray}
Combining \eqref{eqn:M-13} and \eqref{eqn:K-12} together,
we have the estimates \eqref{eqn:M-12}.

Next, assume that $\vec{u}$ is a weak solution
of \eqref{eqn:P-01} in $Q_{4r}=Q_{4r}(X_0)$.
Let $\zeta$ be a smooth cut-off function such that 
$\zeta\equiv 1$ in $Q_r$, vanishes near $\partial_P Q_{2r}$, and satisfies
\begin{equation}
\label{eqn:K-30}
0\le \zeta \le 1\quad\text{and}\quad
\abs{\zeta_t}+\abs{D\zeta}^2 \le Cr^{-2}.
\end{equation}
Note that on each slice $Q_{2r,s}$, we have
\begin{equation*}
\begin{split}
0&=\int_{Q_{2r,s}}\left(\vec{u}_t-D_{\alpha}(\vec{A}^{\alpha\beta}D_\beta
\vec{u})\right)\cdot\zeta^2\vec{u}\\
&=\int_{Q_{2r,s}}\zeta^2\vec{u}_t\cdot\vec{u}+\int_{Q_{2r,s}}\zeta^2
\ip{\vec{A}^{\alpha\beta}D_\beta\vec{u},D_\alpha\vec{u}}+
2\zeta\ip{\vec{A}^{\alpha\beta}D_\beta\vec{u},D_\alpha\zeta\vec{u}}.
\end{split}
\end{equation*}
Using the ellipticity condition and the Cauchy-Schwarz inequality,
we find
\begin{equation*}
\begin{split}
\nu_0\int_{Q_{2r,s}}\zeta^2\abs{D\vec{u}}^2
&\le \int_{Q_{2r,s}}\zeta^2\abs{\vec{u}_t}\abs{\vec{u}}
+C\int_{Q_{2r,s}}\zeta\abs{D\vec{u}}\abs{D\zeta}\abs{\vec{u}}\\
&\le \frac{\epsilon\nu_0}{2}\int_{Q_{2r,s}}\zeta^2\abs{\vec{u}_t}^2+
\frac{C}{\epsilon}\int_{Q_{2r,s}}\zeta^2\abs{\vec{u}}^2+
C\int_{Q_{2r,s}}\abs{D\zeta}^2\abs{\vec{u}}^2 \\
&+\frac{\nu_0}{2}\int_{Q_{2r,s}}\zeta^2\abs{D\vec{u}}^2.
\end{split}
\end{equation*}
Then by \eqref{eqn:K-30}, \eqref{eqn:M-12},
and the energy estimates, for all $s\in[t_0-r^2,t_0]$, we have
\begin{equation}
\label{eqn:M-99}
\begin{split}
\int_{Q_{r,s}}\abs{D\vec{u}}^2
&\le \epsilon\int_{Q_{2r,s}}\abs{\vec{u}_t}^2+
\frac{C}{\epsilon}\int_{Q_{2r,s}}\abs{\vec{u}}^2+
\frac{C}{r^2}\int_{Q_{2r,s}}\abs{\vec{u}}^2 \\
&\le \frac{C\epsilon}{r^6}\int_{Q_{4r}}\abs{\vec{u}}^2+
\frac{C}{\epsilon r^2}\int_{Q_{4r}}\abs{\vec{u}}^2+
\frac{C}{r^4}\int_{Q_{4r}}\abs{\vec{u}}^2.
\end{split}
\end{equation}
If we set $\epsilon=r^2$, then the above estimates \eqref{eqn:M-99}
now become
\begin{equation*}
\int_{Q_{r,s}}\abs{D\vec{u}}^2 \le
\frac{C}{r^4}\int_{Q_{4r}}\abs{\vec{u}}^2,
\end{equation*}
from which the estimates \eqref{eqn:M-11} follows
by a well known covering argument.
\end{proof}

\begin{lemma}
\label{lem:M-03}
Assume that the elliptic system \eqref{eqn:E-01} satisfies
the H\"{o}lder estimates for weak solutions at every scale with constants
$\mu_0, H_0$.
Let $\vec{u}$ be a weak solution of the inhomogeneous elliptic system
\begin{equation}
D_{\alpha}(\vec{A}^{\alpha\beta}(x) D_{\beta}\vec{u})
=\vec{f}\quad\text{in}\quad B_{2}=B_{2}(x_0),
\end{equation}
where $\vec{f}$ belongs to the Morrey space $M^{2,\lambda}(B_{2})$
with $\lambda\ge 0$.
Then, for any $\gamma\ge 0$
with $\gamma<\gamma_0=\min(\lambda+4,n+2\mu_0)$
(we may take $\gamma=\gamma_0$ if $\gamma_0<n$)
there exists
$C=C(n,\nu_0,M_0,\mu_0,H_0,\lambda,\gamma)$ such that
$\vec{u}$ satisfies the following local estimates
\begin{equation}
\label{eqn:M-08}
\int_{B_r(x)}\abs{D\vec{u}}^2
\le C \left( r^{\gamma-2}
\int_{B_2}\abs{D\vec{u}}^2+ r^{\gamma-2}
\norm{\vec{f}}_{M^{2,\lambda}(B_2)}^2
\right)
\end{equation}
uniformly for all $x\in B_1=B_1(x_0)$ and $0<r\le 1$.
Moreover,
if $\gamma<n$,
then $\vec{u}$ belongs to the
Morrey space $M^{2,\gamma}(B_{1})$ and
\begin{equation}
\label{eqn:M-10}
\norm{\vec{u}}_{M^{2,\gamma}(B_{1})}\le 
C\left(\norm{\vec{u}}_{L^2(B_{2})}+
\norm{D\vec{u}}_{L^{2}(B_{2})}+
\norm{\vec{f}}_{M^{2,\lambda}(B_2)}\right).
\end{equation}
\end{lemma}
\begin{proof}
First, we note that the property \eqref{eqn:M-04} implies that
for all $0<\rho<r$ and $x\in\Real^n$, we have
\begin{equation*}
\int_{B_\rho(x)}\abs{D\vec{u}}^2\le C\cdot H_0
\left(\frac{\rho}{r}\right)^{n-2+2\mu_0}
\int_{B_r(x)}\abs{D\vec{u}}^2.
\end{equation*}
In the light of the above observation,
the estimates \eqref{eqn:M-08} is quite standard and is found, for example,
in \cite[Chapter~3]{Giaq83}. Then, by Poincar\'{e} inequality we have
\begin{equation}
\label{eqn:M-07}
\int_{B_r(x)}\abs{\vec{u}-\overline{\vec{u}}_{B_r(x)}}^2
\le C r^{\gamma} \left(
\norm{D\vec{u}}_{L^2(B_2)}^2+
\norm{\vec{f}}_{M^{2,\lambda}(B_2)}^2\right)
\end{equation}
uniformly for all $x\in B_1=B_1(0)$ and $0<r\le 1$.
It is well known that if $\gamma<n$, then the estimates \eqref{eqn:M-07}
yield \eqref{eqn:M-10} (see e.g. \cite[Chapter~3]{Giaq83}).
\end{proof}

%%%%%%%%%%%%%%%%%%%%%%%%%%%%%%%%%%%%%%%%%%%%%%%%%%%%%%%%%%%%%%%%%%%%%%%%%%%%%%%
\subsection{Proof of Theorem~\ref{thm:M-02}}
%%%%%%%%%%%%%%%%%%%%%%%%%%%%%%%%%%%%%%%%%%%%%%%%%%%%%%%%%%%%%%%%%%%%%%%%%%%%%%%
Let
$\vec{u}$ be a weak solution of \eqref{eqn:P-01} in a cylinder $Q_4=Q_4(0)$.
We rewrite \eqref{eqn:P-01} as $L\vec{u}=\vec{u}_t$.
By Lemma~\ref{lem:M-02}, we find that
$\vec{u}_t(\cdot,s)$ is in $L^2(Q_{2,s})$ and 
satisfies
\begin{displaymath}
\norm{\vec{u}_t(\cdot,s)}_{L^2(Q_{2,s})}\le C \norm{\vec{u}}_{L^2(Q_4)}
\quad\text{for all }-4\le s\le 0.
\end{displaymath}
Therefore, we may apply Lemma~\ref{lem:M-03} with $\vec{f}=\vec{u}_t$
and $\lambda=0$, and then apply Lemma~\ref{lem:M-02} to find that
for all $x\in B_{1}(0)$ and $0<r\le 1$, we have
\begin{equation}
\label{eqn:X-11}
\begin{split}
\int_{B_{r}(x)}\abs{D\vec{u}(\cdot,s)}^2
&\le C r^{\gamma-2} \left(
\norm{D\vec{u}(\cdot,s)}_{L^2(Q_{2,s})}^2+
\norm{\vec{u}_t(\cdot,s)}_{L^{2}(Q_{2,s})}^2\right)\\
&\le C r^{\gamma-2} \norm{\vec{u}}_{L^{2}(Q_{4})}^2
\quad\text{uniformly in }s\in[-4,0]
\end{split}
\end{equation}
for all $\gamma<\min(4,n+2\mu_0)$.

By Lemma~\ref{lem:P-02} and then by \eqref{eqn:X-11} we find that
for all $X=(x,t)\in Q_1$ and $r\le 1$
\begin{equation}
\label{eqn:X-12}
\begin{split}
\int_{Q_r(X)}\abs{\vec{u}-\overline{\vec{u}}_{Q_r(X)}}^2
&\le C r^2\int_{t-r^2}^t\int_{B_{r}(x)}
\abs{D\vec{u}(y,s)}^2\,dy\,ds \\
&\le C r^{2+\gamma} \norm{\vec{u}}_{L^{2}(Q_{4})}^2.
\end{split}
\end{equation}
Note that if $n\le 3$, then we may write $\gamma=n+2\mu$ for some $\mu>0$.
In that case, \eqref{eqn:X-12} now reads
\begin{equation}
\int_{Q_r(X)}\abs{\vec{u}-\overline{\vec{u}}_{Q_r(X)}}^2
\le C r^{n+2+2\mu} \norm{\vec{u}}_{L^{2}(Q_{4})}^2
\end{equation}
for all $X\in Q_1$ and $r\le 1$.
Therefore, if $n\le 3$, then Lemma~\ref{lem:P-04} yields the estimates
\begin{equation}
\label{eqn:X-13}
[\vec{u}]_{C^{\mu}_P(Q_{1/2})}
\le C \norm{\vec{u}}_{L^{2}(Q_{4})}.
\end{equation}
We have thus shown that in the case when $n\le 3$, any weak solution
$\vec{u}$ of \eqref{eqn:P-01}
in a cylinder $Q_4=Q_4(0)$ satisfies
the above a priori estimates \eqref{eqn:X-13}
provided that the associated elliptic system
satisfies the H\"{o}lder estimates for weak solutions at every scale.
The general case is recovered as follows.
For given $X_0=(x_0,t_0)$ and $r>0$,
let us consider the new system
\begin{equation}
\label{eqn:scale}
\vec{u}_t-\tilde{L}\vec{u}
:=\vec{u}_t-\sum_{\alpha,\beta=1}^n
D_\alpha(\tilde{\vec{A}}{}^{\alpha\beta}(x) D_\beta\vec{u})=0,
\end{equation}
where $\tilde{\vec{A}}{}^{\alpha\beta}(x)=\vec{A}^{\alpha\beta}(x_0+rx)$.
Note that the associated elliptic system $\tilde{L}\vec{u}=0$
also satisfies the H\"{o}lder estimates for weak solutions at every scale.
Moreover, the ellipticity constants $\nu_0, M_0$ remain the same for
the new coefficients $\tilde{\vec{A}}{}^{\alpha\beta}$.
Let $\vec{u}$ be a weak solution of \eqref{eqn:P-01} in $Q_{4r}(X_0)$.
Then $\tilde{\vec{u}}(X)=\tilde{\vec{u}}(x,t):=\vec{u}(x_0+rx,t_0+r^2t)$
is a weak solution of \eqref{eqn:scale} in $Q_{4}(0)$ and thus
$\tilde{\vec{u}}$ satisfies
the estimates \eqref{eqn:X-13}. By rescaling back to $Q_{4r}(X_0)$,
the estimates \eqref{eqn:X-13} become
\begin{equation}
\label{eqn:X-16}
[\vec{u}]_{C^{\mu}_P(Q_{r/2})}
\le C r^{-(n/2+1+\mu)}\norm{\vec{u}}_{L^{2}(Q_{4r})}.
\end{equation}
Thus, when $n\le 3$, the theorem now
follows from a well known covering argument.

In the case when $n\ge 4$, we invoke a bootstrap argument.
For the sake of simplicity, let us momentarily assume that $4\le n \le 7$.
Let $\vec{u}$ be a weak solution of \eqref{eqn:P-01} in
$Q_{8}=Q_{8}(0)$.
Let us fix $X_0=(x_0,t_0)\in Q_2(0)$ and observe that
$\vec{u}_t$ also satisfies the system \eqref{eqn:P-01} in $Q_{4}(X_0)$.
Thus, by a similar argument that led to \eqref{eqn:X-11}, we find
that for all $x\in B_{1}(x_0)$ and $0<r\le 1$ we have
\begin{equation}
\int_{B_{r}(x)}\abs{D\vec{u}_t(\cdot,s)}^2
\le C r^{\gamma-2} \norm{\vec{u}_t}_{L^{2}(Q_{4}(X_0))}^2
\quad\text{uniformly in }s\in [t_0-4,t_0],
\end{equation}
for all $\gamma<4$ (we may take $\gamma=4$ if $n>4$).
Then, by \eqref{eqn:M-10} in Lemma~\ref{lem:M-03},
Lemma~\ref{lem:M-01}, and Lemma~\ref{lem:M-02} we conclude that
\begin{equation}
\label{eqn:X-14}
\norm{\vec{u}_t(\cdot,s)}_{M^{2,\gamma}(B_{1}(x_0))}
\le C \norm{\vec{u}}_{L^{2}(Q_{8}(0))}
\quad\text{for all }s\in [t_0-4,t_0].
\end{equation}
Since the above estimates \eqref{eqn:X-14} hold
for all $X_0=(x_0,t_0)\in Q_{2}(0)$,
we find that, in particular,
$\vec{u}_t(\cdot, s)$ belongs to $M^{2,\gamma}(B_{2}(0))$
for all $-4\le s \le 0$,
and satisfies
\begin{equation}
\label{eqn:X-15}
\norm{\vec{u}_t(\cdot,s)}_{M^{2,\gamma}(B_{2}(0))}
\le C \norm{\vec{u}}_{L^{2}(Q_{8}(0))}
\quad\text{for all }s\in [-4,0],
\end{equation}
where we also used \eqref{eqn:M-12} of Lemma~\ref{lem:M-02}.

The above estimates \eqref{eqn:X-15} for $\vec{u}_t$ now 
allows us to invoke Lemma~\ref{lem:M-03} with $\vec{f}=\vec{u}_t$ and
$\lambda=\gamma$. Then, by Lemma~\ref{lem:M-03} and Lemma~\ref{lem:M-02},
we find that for all $x\in B_{1}(0)$ and $0<r\le 1$, we have
\begin{equation*}
\int_{B_{r}(x)}\abs{D\vec{u}(\cdot,s)}^2
\le C r^{\overline{\gamma}-2} \norm{\vec{u}}_{L^{2}(Q_{8}(0))}^2
\quad\text{uniformly in }s\in [-4,0]
\end{equation*}
for all $\overline{\gamma}<\min(\gamma+4,n+2\mu_0)$.
Since we assume that $n\le 7$, we may write
$\overline{\gamma}=n+2\overline{\mu}$ for
some $\overline{\mu}>0$.
By the exactly same argument we used in the case when $n\le 3$,
we derive the estimates
\begin{equation*}
[\vec{u}]_{C^{\mu}_P(Q_{1/2})}
\le C \norm{\vec{u}}_{L^{2}(Q_{8})},
\end{equation*}
and the theorem follows as before.

Finally, if $n\ge 8$, we repeat the above process;
if $\vec{u}$ is a weak solution of \eqref{eqn:P-01} in $Q_{16}(0)$,
then $\vec{u}_t(\cdot,s)$ is
in $M^{2,\gamma}(B_{1}(0))$ for all $\gamma<8$ and so on.
The process cannot go on indefinitely and
it stops in $k=[n/4]+1$ steps. The proof is complete.
\hfil\qed

%%%%%%%%%%%%%%%%%%%%%%%%%%%%%%%%%%%%%%%%%%%%%%%%%%%%%%%%%%%%%%%%%%%%%%%%%%%%%%%
\subsection{Proof of Theorem~\ref{thm:M-03}}
%%%%%%%%%%%%%%%%%%%%%%%%%%%%%%%%%%%%%%%%%%%%%%%%%%%%%%%%%%%%%%%%%%%%%%%%%%%%%%%
The proof is based on Theorem~\ref{thm:P-01}, the proof of which, in turn,
is found in \cite{HK}.
By Theorem~\ref{thm:P-01}, we only need to establish the local boundedness
property for weak solutions of the parabolic system \eqref{eqn:P-01}
and for those of its adjoint system \eqref{eqn:P-02}.

From the hypothesis that the elliptic system \eqref{eqn:E-01}
satisfies the H\"{o}lder estimates for weak solutions at every scale,
we find, by Theorem~\ref{thm:M-02}, that the parabolic system
\eqref{eqn:P-01}
with the associated time-independent coefficients
also satisfies the H\"{o}lder estimates
for weak solutions at every scale; that is,
there exist some constants $\mu>0$ and $C$,
depending on the prescribed quantities,
such that if $\vec{u}$ is a
weak solution of \eqref{eqn:P-01} in $Q_{4r}(X)$, then it satisfies 
the estimates
\begin{equation*}
[\vec{u}]_{C^{\mu}_P(Q_{2r})}\le C r^{-(n/2+1+\mu)}
\norm{\vec{u}}_{L^2(Q_{4r})}.
\end{equation*}
Let us fix $Y\in Q_r=Q_{r}(X)$.
Then, for all $Z\in Q_{r}(Y)\subset Q_{2r}(X)$, we have
\begin{equation}
\label{eqn:final}
\abs{\vec{u}(Y)}
\le \abs{\vec{u}(Z)}+
d_P(Y,Z)^{\mu}\cdot
[\vec{u}]_{C^{\mu}_P(Q_{2r})}
\le \abs{\vec{u}(Z)}+ C r^{-(n/2+1)}
\norm{\vec{u}}_{L^2(Q_{4r})}.
\end{equation}
By averaging \eqref{eqn:final} over $Q_{r}(Y)$ with respect to $Z$,
we derive
(note $\abs{Q_r}=Cr^{n+2}$)
\begin{equation*}
\abs{\vec{u}(Y)}
\le C r^{-(n+2)} \norm{\vec{u}}_{L^1(Q_r(Y))}
+ C r^{-(n/2+1)} \norm{\vec{u}}_{L^2(Q_{4r})}.
\end{equation*}
Since $Y\in Q_r$ is arbitrary,
we find, by H\"{o}lder's inequality, that $\vec{u}$ satisfies
\begin{equation*}
\norm{\vec{u}}_{L^\infty(Q_r)}
\le C r^{-(n/2+1)} \norm{\vec{u}}_{L^2(Q_{4r})}
\end{equation*}
for some constant $C=C(n,\nu_0,M_0,\mu_0,H_0)$.

To finish the proof, we also need to show that
if $\vec{u}$ is a weak solution
of the adjoint system \eqref{eqn:P-02} in $Q^{*}_{4r}=Q^{*}_{4r}(X)$,
then it satisfies
the local boundedness property
\begin{equation}
\label{last}
\norm{\vec{u}}_{L^\infty(Q^{*}_r)}
\le C r^{-(n/2+1)} \norm{\vec{u}}_{L^2(Q^{*}_{4r})}.
\end{equation}
The verification of \eqref{last} requires
only a slight modification of the previous
arguments (mostly, one needs to replace $Q_r$ by $Q_r^{*}$ and so on),
but it is rather routine and we skip the details.
\hfil\qed

%%%%%%%%%%%%%%%%%%%%%%%%%%%%%%%%%%%%%%%%%%%%%%%%%%%%%%%%%%%%%%%%%%%%%%%%%%%%%%%
\bibliographystyle{amsplain}

\end{document}